\documentclass[12pt, fleqn]{amsart}
\usepackage[margin=3cm]{geometry}

\usepackage[cp1250]{inputenc}
\usepackage{enumitem, graphicx, caption, lpic}
\usepackage[sc]{mathpazo}
\usepackage{amssymb, amsthm, amsmath}

\theoremstyle{plain}
\newtheorem{thm}{Theorem}
\newtheorem{prop}[thm]{Proposition}
\newtheorem{lem}[thm]{Lemma}
\newtheorem{cor}[thm]{Corollary}

\theoremstyle{definition}
\newtheorem{defn}[thm]{Definition}
\newtheorem{exmp}[thm]{Example}
\newtheorem{remark}[thm]{Remark}

\allowdisplaybreaks

\usepackage{caption}

\usepackage{fancyhdr}
\begin{document}

\title[Presentations and representations of surface singular braids]{Presentations and representations\\ of surface singular braid monoids}
\author[M. Jab{\l}onowski]{Micha{\l} Jab{\l}onowski}
\dedicatory{Dedicated to Ma{\l}gorzata}
\address{Institute of Mathematics, University of Warsaw, 02-097 Warsaw, Poland,\\ and\\ Institute of Mathematics, Faculty of Mathematics, Physics and Informatics,\\ University of Gda\'nsk, 80-308 Gda\'nsk, Poland}
\email{michal.jablonowski@gmail.com}

\thanks{The author's research was partially supported by grant BW-UG 538-5100-B297-16 for young researchers and 
Warsaw Center of Mathematics and Computer Science.}

\begin{abstract}
The surface singular braid monoid corresponds to marked graph diagrams of knotted surfaces in braid form. In a quest to resolve linearity problem for this monoid, we will show that if it is defined on at least two or at least three strands, then its two or respectively three dimensional representations are not faithful. We will also derive new presentations for the surface singular braid monoid, one with reduced the number of defining relations, and the other with reduced the number of its singular generators. We include surface singular braid formulations of all knotted surfaces in Yoshikawa's table.

\end{abstract}

\maketitle

\section{Introduction}

The well known Artin representation of the braid group $B_n$ may be used to calculate the group of a knot. Applying Fox' free differential calculus to this representation, we can derive the Burau representation. Its irreducible part may be used to calculate the Alexander polynomial of a knot. In \cite{Gem01}, B. Gemein extend the Artin and the Burau representation to a representation of the Baez-Birman singular braid monoid $SB_n$. A monoid is said to be linear if it is isomorphic to a submonoid of matrices $M_n(K)$ for some natural number $n$ and some field $K$. In \cite{DasGem00}, O. T. Dasbach and B. Gemein showed the faithfulness of the two dimensional extended Burau representation of $SB_3$, therefore this monoid is linear. 

It is natural then to search for a faithful representation of the surface singular braid monoid $SSB_n$ defined in \cite{Jab13}, where the author classified knotted surfaces in $\mathbb{R}^4$ that have surface singular braid index equal to one or two, and also showed that there exist infinitely many surface-link types that are closures of elements from $SSB_3$. We will show in this paper that any representation of $SSB_n$, for any $n\geqslant 3$, to the multiplicative monoid of all $2\times 2$ or $3\times 3$ matrices with entries in a given field, is not faithful. We will also derive new presentations for the surface singular braid monoid, one with reduced the number of defining relations, and the other with reduced the number of its defining non-classical generators. 

\section{Marked graph diagrams}

An embedding (or its image) of a closed (i.e., compact, without boundary) surface into $\mathbb{R}^4$ is called a \emph{knotted surface} (or \emph{surface-link}). Two knotted surfaces are \emph{equivalent} (or have the same \emph{type}) if there exists an orientation preserving homeomorphism of the four-space $\mathbb{R}^4$ to itself (or equivalently auto-homeomorphism of the four-sphere $\mathbb{S}^4$), mapping one of those surfaces onto the other. We will work in the standard smooth category. Let $\mathbb{R}^3_t$ denote $\mathbb{R}^3\times\{t\}$ for $t\in\mathbb{R}$. 

It is well known (\cite{Lom81}, \cite{KSS82}, \cite{Kam89}) that for any knotted surface $F$, there exists a surface-link $F'$ satisfying the following: $F'$ is equivalent to $F$ and has only finitely many Morse's critical points, all maximal points of $F'$ lie in $\mathbb{R}^3_1$, all minimal points of $F'$ lie in $\mathbb{R}^3_{-1}$, all saddle points of $F'$ lie in $\mathbb{R}^3_0$.

The zero section $\mathbb{R}^3_0\cap F'$ of the surface F' gives us then a 4-valent graph. We assign to each vertex a \emph{marker} that informs us about one of the two possible types of saddle points (see Fig.\;\ref{pic03}) depending on the shape of $\mathbb{R}^3_{-\epsilon}\cap F'$ or $\mathbb{R}^3_{\epsilon}\cap F'$ for a small real number $\epsilon>0$. The resulting graph is called a \emph{marked graph}.

Making now a projection in general position of this graph to $\mathbb{R}^2$ and assigning types of classical crossings between regular arcs, we obtain a \emph{marked graph diagram}. For a marked graph diagram $D$, we denote by $L_+(D)$ and $L_-(D)$ the diagrams obtained from $D$ by smoothing every vertex as presented in Fig.\;\ref{pic03} for $+\epsilon$ and $-\epsilon$, respectively.
\begin{figure}[ht]
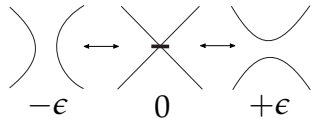

\begin{center}
\begin{lpic}[b(0.5cm)]{./pic03(4cm)}
	\lbl[t]{15,-2;$-\epsilon$}
  \lbl[t]{60,-2;$0$}
  \lbl[t]{102,-2;$+\epsilon$}
	\end{lpic}
		\caption{Rules for smoothing a marker.\label{pic03}}
\end{center}
\end{figure}
\begin{thm}[\cite{Swe01}, \cite{KeaKur08}, \cite{JKL15}]
Any two marked graph diagrams representing the same type of knotted surface are related by a finite sequence of Yoshikawa local moves presented in Fig.\;\ref{pic04} (and an isotopy of the diagram in $\mathbb{R}^2$).
\end{thm}
\begin{figure}[ht]
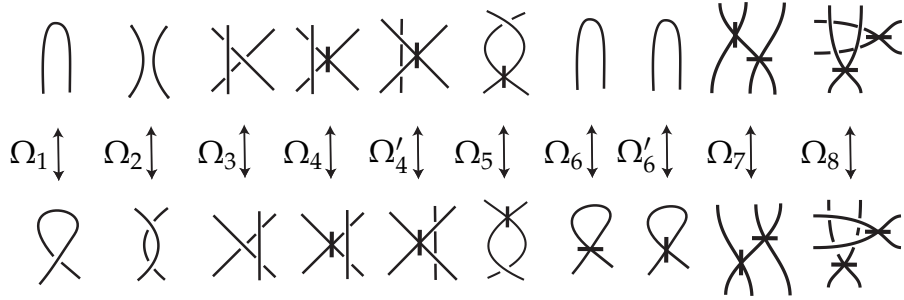

\begin{center}
\begin{lpic}[l(0.2cm)]{./pic04(11.5cm)}
	\lbl[r]{3,33;$\Omega_1$}
  \lbl[r]{25,33;$\Omega_2$}
	\lbl[r]{47,33;$\Omega_3$}
	\lbl[r]{67,33;$\Omega_4$}
	\lbl[r]{87,33;$\Omega_4'$}
	\lbl[r]{107,33;$\Omega_5$}
	\lbl[r]{128,33;$\Omega_6$}
	\lbl[r]{145,33;$\Omega_6'$}
	\lbl[r]{166,33;$\Omega_7$}
	\lbl[r]{188,33;$\Omega_8$}
	\end{lpic}
\caption{A generating set of Yoshikawa moves (compare \cite{Yos94}).\label{pic04}}
\end{center}
\end{figure}
\section{Surface singular braid monoid}

We can present every marked diagram of a surface-link in a \emph{braid form} defined as the geometric closure of a singular braid with markers. We have the monoid $SSB_m$ that corresponds to marked graph diagrams in braid form on $m$ strands. For $m=1$ this monoid is trivial with one element, let us assume that $m>1$. Elements of $SSB_m$, called \emph{surface singular braids}, are generated by four types of elements $a_i, b_i, c_i, c_i^{-1}$ for $i=1, \ldots, m-1$, where the correspondence of types of crossings and types of markers between $i$-th and $i+1$-th strand (in the horizontal position, numbered from the top to the bottom) is presented in Fig.\;\ref{pic05}.
\begin{figure}[ht]
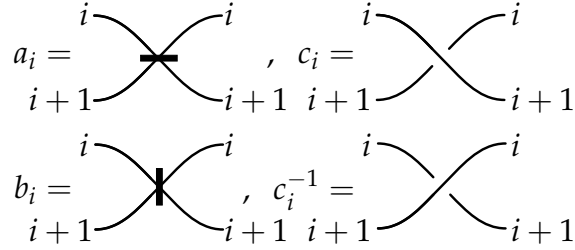

\begin{center}
\begin{lpic}[]{./pic05(5.5cm)}
	\lbl[r]{-5,42;$a_i=$}
	\lbl[r]{-5,10;$b_i=$}
	\lbl[r]{64,42;$,\;\;c_i=$}
	\lbl[r]{64,10;$,\;\;c_i^{-1}=$}
	\lbl[r]{-1,53;$i$}
	\lbl[r]{-1,21;$i$}
	\lbl[r]{0,32;$i+1$}
	\lbl[r]{0,0;$i+1$}
	\lbl[r]{68,53;$i$}
	\lbl[r]{68,21;$i$}
	\lbl[r]{68,32;$i+1$}
	\lbl[r]{68,0;$i+1$}
	\lbl[l]{32,53;$i$}
	\lbl[l]{32,21;$i$}
	\lbl[l]{32,32;$i+1$}
	\lbl[l]{32,0;$i+1$}
	\lbl[l]{102,53;$i$}
	\lbl[l]{102,21;$i$}
	\lbl[l]{102,32;$i+1$}
	\lbl[l]{102,0;$i+1$}
	\end{lpic}
	\caption{The correspondence of monoid generators.\label{pic05}}
\end{center}
\end{figure}
\begin{defn}[\cite{Jab13}]
Let $m\in \mathbb{Z}$, $m>1$ and $i,k,n\in\{1, \ldots, m-1\}$ such that $|k-i|=1$, moreover let $x_i,y_i\in\{a_i,b_i,c_i, c_i^{-1}\}$.
Monoid $SSB_m$ is subject to the following relations.
\begin{enumerate}
\item[(A1)] $c_ic_i^{-1}=1$,
\item[(A2)] $x_iy_n=y_nx_i$ for $n\not=k$,
\item[(A3)] $x_ic_kc_i=c_kc_ix_k$,
\item[(A4)] $x_ic_k^{-1}c_i^{-1}=c_k^{-1}c_i^{-1}x_k$,
\item[(A5)] $a_ib_k=b_ka_i$,
\item[(A6)] $a_ib_{i-2}(c_{i-1}c_{i-2}c_ic_{i-1})^2=a_ib_{i-2}$ for $i>2$,
\item[(A7)] $b_ia_{i-2}(c_{i-1}c_{i-2}c_ic_{i-1})^2=b_ia_{i-2}$ for $i>2$,
\item[(A8)] $a_i^2=a_i$,
\item[(A9)] $b_i^2=b_i$,
\item[(A10)] $a_ib_ic_i^2=a_ib_i$,
\item[(A11)] $a_ib_k(c_ic_kc_i)^2=a_ib_k$.
\end{enumerate}
\end{defn}
We will indicate our closure of a marked graph diagram in a braid form by adding square brackets around its words and adding lower index after it, saying how many strands we are joining. Let us further denote by $CSB_m$ a subset of $SSB_m$ containing only those elements $x$ such that $L_+([x]_m)$ and $L_-([x]_m)$ are diagrams of trivial classical links. We define the following additional relations on closed braids.
\begin{enumerate}
\item[(C1)] $\left[x_iS_n\right]_n=\left[S_nx_i\right]_n$ for $n\in\mathbb{Z}_+$ and $i<n$ and $x_iS_n\in CSB_n$,
\item[(C2)] $\left[S_n\right]_n=\left[S_nx_n\right]_{n+1}$ for $n\in\mathbb{Z}_+$ and $S_n\in CSB_n$.
\end{enumerate}
\begin{thm}[\cite{Jab13}]\label{thm4}
Making change in a closed braid word formulation of a knotted surface by using one of relations from (A1)-(A11) or (C1)-(C2), we receive a formula of a knotted surface of the same type.
\end{thm}
\begin{prop}\label{prop2}
The monoid $SSB_m$ for $m\in \mathbb{Z}$ and $m>1$ is generated by $a_i,b_i,c_i,c_i^{-1}$ for $i,j\in\{1,\ldots,m-1\}$, $x_i,y_i\in\{a_i,b_i,c_i, c_i^{-1}\}$ and is subject to the following relations:
\begin{align}
c_ic_i^{-1}&=1=c_i^{-1}c_i,& \tag{R1}\\
x_iy_j&=y_jx_i& \text{\;\; for } |i-j|>1, \tag{R2}\\
a_ic_i&=c_ia_i,& \tag{R3}\\
b_ic_i&=c_ib_i,& \tag{R4}\\
c_{i+1}c_ic_{i+1}&=c_ic_{i+1}c_i& \text{\;\; for } i<m-1, \tag{R5}\\
a_{i+1}c_ic_{i+1}&=c_ic_{i+1}a_i& \text{\;\; for } i<m-1, \tag{R6}\\
b_{i+1}c_ic_{i+1}&=c_ic_{i+1}b_i& \text{\;\; for } i<m-1, \tag{R7}\\
a_{i}c_{i+1}c_{i}&=c_{i+1}c_{i}a_{i+1}& \text{\;\; for } i<m-1, \tag{R8}\\
b_{i}c_{i+1}c_{i}&=c_{i+1}c_{i}b_{i+1}& \text{\;\; for } i<m-1, \tag{R9}\\
a_ib_{i+1}&=b_{i+1}a_i& \text{\;\; for } i<m-1, \tag{R10}\\
a_ib_i&=b_ia_i,& \tag{R11}\\
a_i^2&=a_i,& \tag{R12}\\
b_i^2&=b_i,& \tag{R13}\\
a_ib_ic_i^2&=a_ib_i,& \tag{R14}\\
a_ib_{i+1}(c_ic_{i+1}c_i)^2&=a_ib_{i+1}& \text{\;\; for } i<m-1, \tag{R15}\\
a_ib_{i+2}(c_{i+1}c_ic_{i+2}c_{i+1})^2&=a_ib_{i+2}& \text{\;\; for } i<m-2. \tag{R16}
\end{align}
\end{prop}
\begin{proof}
Some relations from (A2)-(A4) that includes $c_i^{-1}$ are known, from classical singular braid theory, to follow from (R1)-(R9). The remaining relations are either the same or derived as follows.
\begin{align}
\noindent
\begin{aligned}
\hspace*{-1cm}b_ia_{i+1}\overset{\scriptscriptstyle\text{(R1)}}{=}&b_ia_{i+1}c_{i}c_{i+1}c_{i+1}^{-1}c_{i}^{-1}\overset{\scriptscriptstyle\text{(R6)}}{=}b_ic_{i}c_{i+1}a_{i}c_{i+1}^{-1}c_{i}^{-1}\overset{\scriptscriptstyle\text{(R1),(R3),(R4)}}{=}\\
&c_ib_ic_{i+1}c_ia_ic_i^{-1}c_{i+1}^{-1}c_i^{-1}\overset{\scriptscriptstyle\text{(R9)}}{=}c_ic_{i+1}c_ib_{i+1}a_ic_i^{-1}c_{i+1}^{-1}c_i^{-1}\overset{\scriptscriptstyle\text{(R10)}}{=}\\
&c_ic_{i+1}c_ia_ib_{i+1}c_i^{-1}c_{i+1}^{-1}c_i^{-1}\overset{\scriptscriptstyle\text{(R3)}}{=}c_ic_{i+1}a_ic_ib_{i+1}c_i^{-1}c_{i+1}^{-1}c_i^{-1}\overset{\scriptscriptstyle\text{(R6)}}{=}\\
&a_{i+1}c_ic_{i+1}c_ib_{i+1}c_i^{-1}c_{i+1}^{-1}c_i^{-1}\overset{\scriptscriptstyle\text{(R4),(R9)}}{=}a_{i+1}b_ic_ic_{i+1}c_ic_i^{-1}c_{i+1}^{-1}c_i^{-1}\overset{\scriptscriptstyle\text{(R1)}}{=}a_{i+1}b_i,
\end{aligned}
\end{align}
\begin{align}
\begin{aligned}
\hspace*{-1cm}a_{i+1}b_i(c_ic_{i+1}c_i)^2\overset{\scriptscriptstyle\text{(R5)}}{=}&a_{i+1}b_ic_{i+1}c_{i}c_{i+1}c_{i}c_{i+1}c_{i}\overset{\scriptscriptstyle\text{(R3),(R4),(R9)}}{=}c_{i+1}a_{i+1}c_{i}c_{i+1}b_{i+1}c_{i}c_{i+1}c_{i}\\
&\overset{\scriptscriptstyle\text{(R1),(R6)}}{=}c_{i+1}c_{i}c_{i+1}a_{i}b_{i+1}(c_{i}c_{i+1}c_{i})^2(c_{i}c_{i+1}c_{i})^{-1}\overset{\scriptscriptstyle\text{(R15)}}{=}\\
&c_{i+1}c_{i}c_{i+1}a_{i}b_{i+1}(c_{i}c_{i+1}c_{i})^{-1}\overset{\scriptscriptstyle\text{(R3)-(R6),(R9)}}{=}\\
&a_{i+1}b_{i}c_{i}c_{i+1}c_{i}c_{i}^{-1}c_{i+1}^{-1}c_{i}^{-1}\overset{\scriptscriptstyle\text{(R1)}}{=}a_{i+1}b_{i},
\end{aligned}
\end{align}
\begin{align}
\begin{aligned}
\hspace*{-1cm}a_{i+2}b_{i}(c_{i+1}c_ic_{i+2}c_{i+1})^2\overset{\scriptscriptstyle\text{(R9)}}{=}&a_{i+2}c_{i+1}c_ic_{i+2}c_{i+1}b_{i+2}c_{i+1}c_ic_{i+2}c_{i+1}\overset{\scriptscriptstyle\text{(R2),(R6)}}{=}\\
&c_{i+1}c_ic_{i+2}c_{i+1}a_{i}b_{i+2}(c_{i+1}c_ic_{i+2}c_{i+1})\\
&\overset{\scriptscriptstyle\text{(R16)}}{=}c_{i+1}c_ic_{i+2}c_{i+1}a_{i}b_{i+2}c_{i+1}^{-1}c_{i+2}^{-1}c_{i}^{-1}c_{i+1}^{-1}\overset{\scriptscriptstyle\text{(R2),(R6),(R9)}}{=}\\
&a_{i+2}b_{i}c_{i+1}c_ic_{i+2}c_{i+1}c_{i+1}^{-1}c_{i+2}^{-1}c_{i}^{-1}c_{i+1}^{-1}\overset{\scriptscriptstyle\text{(R1)}}{=}a_{i+2}b_{i}.
\end{aligned}
\end{align}
\end{proof}
Sometimes (for computational reasons) we want to have less generators and therefore the following presentation is useful.
\begin{prop}\label{prop3}
The monoid $SSB_n$ for $n\in \mathbb{Z}$ and $n>1$ is generated by $a,b$ and $c_i,c_i^{-1}$ for $i=1,\ldots,n-1$ and is subject to the following relations:
\begin{align}
c_ic_i^{-1}&=1=c_i^{-1}c_i, \tag{m1}\\
c_ic_j&=c_jc_i& \text{\;\; for } i+1<j<n, \tag{m2}\\
c_ic_{i+1}c_i&=c_{i+1}c_ic_{i+1}& \text{\;\; for } i<n-1, \tag{m3}\\
ac_i&=c_ia& \text{\;\; for } i\not=2, \tag{m4}\\
bc_i&=c_ib& \text{\;\; for } i\not=2, \tag{m5}\\
ac_2c_1^2c_2&=c_2c_1^2c_2a, \tag{m6}\\
bc_2c_1^2c_2&=c_2c_1^2c_2b, \tag{m7}\\
(ac_2c_3c_1c_2)^2&=(c_2c_3c_1c_2a)^2, \tag{m8}\\
(bc_2c_3c_1c_2)^2&=(c_2c_3c_1c_2b)^2, \tag{m9}\\
ac_2bc_2^{-1}&=c_2bc_2^{-1}a, \tag{m10}\\
ab&=ba, \tag{m11}\\
a^2&=a, \tag{m12}\\
b^2&=b, \tag{m13}\\
ac_1b&=ac_1^{-1}b, \tag{m14}\\
a(c_1c_2c_1)b&=a(c_1c_2c_1)^{-1}b, \tag{m15}\\
a(c_2c_3c_1c_2)b&=a(c_2c_3c_1c_2)^{-1}b. \tag{m16}
\end{align}
\end{prop}
\begin{proof}
Set $a=a_1,b=b_1$ and introduce 
$$a_{i+1}=c_{i}c_{i+1}a_{i}c_{i+1}^{-1}c_{i}^{-1},\; b_{i+1}=c_{i}c_{i+1}b_{i}c_{i+1}^{-1}c_{i}^{-1}$$ 
for $i>1$ from the relations (R6), (R7). The relations (R1), (R5) are the same as (m1), (m3) respectively. From the proof of Prop.\;2.2 in \cite{DasGem00} (when $\tau$ is replaced here either by $a$ or $b$, and $\sigma$ is replaced by $c$), it follows that:

\begin{enumerate}
\item the relations $a_1a_3=a_3a_1$, $b_1b_3=b_3b_1$ (part of (R2)) follow from the relations (m1), (m8), (m9), (R6), (R7),
\item other relations from (R2) follow from (m1)-(m5), (R6), (R7),
\item the relations (R3), (R4) follow from (m1)-(m3), (R6), (R7),
\item the relations (R8), (R9) follow from (m1), (m6), (m7), (R6), (R7).
\end{enumerate}

For $i=1$, the relations (m10), (m14)-(m16) are easily equivalent to (R10), (R14)-(R16) respectively. Moreover, for $i=1$ the relations (R11)-(R13) are the same as (m11)-(m13) respectively. We now derive the relations (R10)-(R16) for $i>1$.
The relation (R10) follows from (m10) and the following inductive step
\begin{align}
\begin{aligned}\notag
\hspace*{-1cm}a_ib_{i+1}&=c_{i-1}c_ia_{i-1}c_i^{-1}c_{i-1}^{-1}c_{i-1}b_{i+1}c_{i-1}^{-1}=c_{i-1}c_ia_{i-1}c_i^{-1}c_ic_{i+1}b_ic_{i+1}^{-1}c_i^{-1}c_{i-1}^{-1}\\
&=c_{i-1}c_ic_{i+1}a_{i-1}b_ic_{i+1}^{-1}c_i^{-1}c_{i-1}^{-1}\overset{\text{ind.}}{=}c_{i-1}c_ic_{i+1}b_ia_{i-1}c_{i+1}^{-1}c_i^{-1}c_{i-1}^{-1}\\
&=b_{i+1}c_{i-1}c_ic_{i+1}c_{i+1}^{-1}c_i^{-1}c_{i-1}^{-1}a_i=b_{i+1}a_i.
\end{aligned}
\end{align}
The relation (R11) follows from (m11) and the following inductive step
\begin{align}\notag
\begin{aligned}
\hspace*{-1cm}a_ib_i&=a_ic_{i-1}c_{i}c_{i}^{-1}c_{i-1}^{-1}b_i=c_{i-1}c_{i}a_{i-1}b_{i-1}c_{i}^{-1}c_{i-1}^{-1}\overset{\text{ind.}}{=}c_{i-1}c_{i}b_{i-1}a_{i-1}c_{i}^{-1}c_{i-1}^{-1}\\
&=b_ic_{i-1}c_{i}c_{i}^{-1}c_{i-1}^{-1}a_i=b_ia_i.
\end{aligned}
\end{align}
The relation (R12) follows from (m12) and the following inductive step
\begin{align}\notag
\begin{aligned}
\hspace*{-3cm}a_i^2&=a_ia_i=a_ic_{i-1}c_{i}c_{i}^{-1}c_{i-1}^{-1}a_i=c_{i-1}c_{i}a_{i-1}a_{i-1}c_{i}^{-1}c_{i-1}^{-1}\overset{\text{ind.}}{=}c_{i-1}c_{i}a_{i-1}c_{i}^{-1}c_{i-1}^{-1}\\
&=a_ic_{i-1}c_{i}c_{i}^{-1}c_{i-1}^{-1}=a_i.
\end{aligned}
\end{align}
The relation (R13) follows from (m13) by the similar argument as in the relation (R12).\\ 
The relation (R14) follows from (m14) and the following inductive step
\begin{align}\notag
\begin{aligned}
\hspace*{-1cm}a_ib_ic_i^2&=a_ib_ic_{i-1}c_{i-1}^{-1}c_ic_{i-1}c_{i-1}^{-1}c_i=a_ib_ic_{i-1}c_{i}c_{i-1}c_i^{-1}c_{i-1}^{-1}c_i\\
&=c_{i-1}c_{i}a_{i-1}b_{i-1}c_{i-1}^2c_i^{-1}c_{i-1}^{-1}\overset{\text{ind.}}{=}c_{i-1}c_{i}a_{i-1}b_{i-1}c_i^{-1}c_{i-1}^{-1}\\
&=a_{i}b_{i}c_{i-1}c_{i}c_i^{-1}c_{i-1}^{-1}=a_ib_i.
\end{aligned}
\end{align}
The relation (R15) follows from (m15) and the following inductive step
\begin{align}\notag
\begin{aligned}
\hspace*{-1cm}a_ib_{i+1}(c_ic_{i+1}c_i)^2&=c_{i-1}c_ia_{i-1}c_i^{-1}c_{i-1}^{-1}c_{i-1}b_{i+1}c_{i-1}^{-1}(c_ic_{i+1}c_i)^2\\
&=c_{i-1}c_ia_{i-1}c_i^{-1}c_ic_{i+1}b_ic_{i+1}^{-1}c_i^{-1}c_{i-1}^{-1}(c_ic_{i+1}c_i)(c_ic_{i+1}c_i)\\
&=c_{i-1}c_ic_{i+1}a_{i-1}b_ic_{i+1}^{-1}c_{i-1}c_i^{-1}c_{i-1}^{-1}c_{i+1}c_i(c_ic_{i+1}c_i)\\
&=c_{i-1}c_ic_{i+1}a_{i-1}b_ic_{i-1}c_ic_{i+1}^{-1}c_i^{-1}c_{i-1}^{-1}c_i(c_ic_{i+1}c_i)\\
&=c_{i-1}c_ic_{i+1}a_{i-1}b_i(c_{i-1}c_{i}c_{i-1})c_{i+1}^{-1}c_i^{-1}c_{i-1}^{-1}(c_ic_{i+1}c_i)\\
&=c_{i-1}c_ic_{i+1}a_{i-1}b_i(c_{i-1}c_{i}c_{i-1})^2c_{i+1}^{-1}c_i^{-1}c_{i-1}^{-1}\\
&\overset{\text{ind.}}{=}c_{i-1}c_ic_{i+1}a_{i-1}b_ic_{i+1}^{-1}c_i^{-1}c_{i-1}^{-1}\\
&=a_ic_{i-1}c_ic_{i+1}c_{i+1}^{-1}c_i^{-1}c_{i-1}^{-1}b_{i+1}=a_ib_{i+1}.
\end{aligned}
\end{align}
The relation (R16) follows from (m16) and the following inductive step
\begin{align}\notag
\begin{aligned}
\hspace*{-1.3cm}a_ib_{i+2}(c_{i+1}c_ic_{i+2}c_{i+1})^2&=c_{i-1}c_ia_{i-1}c_i^{-1}c_{i-1}^{-1}c_{i-1}c_ib_{i+2}c_i^{-1}c_{i-1}^{-1}(c_{i+1}c_ic_{i+2}c_{i+1})^2\\
&=c_{i-1}c_ia_{i-1}c_{i+1}c_{i+2}b_{i+1}c_{i+2}^{-1}c_{i+1}^{-1}c_i^{-1}c_{i-1}^{-1}(c_{i+1}c_ic_{i+2}c_{i+1})^2\\
&=c_{i-1}c_ic_{i+1}c_{i+2}a_{i-1}b_{i+1}(c_ic_{i-1}c_{i+1}c_i)^2c_{i+2}^{-1}c_{i+1}^{-1}c_i^{-1}c_{i-1}^{-1}\\
&\overset{\text{ind.}}{=}c_{i-1}c_ic_{i+1}c_{i+2}a_{i-1}b_{i+1}c_{i+2}^{-1}c_{i+1}^{-1}c_i^{-1}c_{i-1}^{-1}\\
&=a_ic_{i-1}c_ic_{i+1}c_{i+2}c_{i+2}^{-1}c_{i+1}^{-1}c_i^{-1}c_{i-1}^{-1}b_{i+2}=a_ib_{i+2}.
\end{aligned}
\end{align}
\end{proof}
\begin{prop}[\cite{Jab13}]\label{prop1}
We have the following all (un)knotted surfaces whose surface singular braids can be defined with two strands
$\mathbb{P}^2_+=[ac_1]_2,$ $\mathbb{P}^2_-=[ac_1^{-1}]_2,$ $\mathbb{T}^2=[ab]_2,$ $\mathbb{KB}^2=[abc_1]_2,$ $\mathbb{S}^2=[c_1]_2,$ $\mathbb{S}^2\sqcup\mathbb{S}^2=[1]_2.$
The $n$-twist-spun surface-knot of the classical rational link $C[k_1,k_2,\ldots,k_{2m+1}]$ in Conway notation encodes as 
$$\tau^n(C[k_1,k_2,\ldots,k_{2m+1}])=[ac_2^{k_{2m+1}}c_1^{-k_{2m}}\cdots c_2^{k_1}bc_2^{-k_1}c_1^{k_2}\cdots c_2^{-k_{2m+1}}(c_1c_2c_1)^{2n}]_3.$$
\end{prop}

Some of knotted surfaces in Yoshikawa's table are included in the above case as follows: $6_1^{0,1}=\tau^0(C[2])$, $8_1=\tau^0(C[3])$, $10_1=\tau^0(C[2,1,1])$, $10_2=\tau^2(C[3])$, $10_3=\tau^3(C[3])$, $10_1^{0,1}=\tau^0(C[4])$. These and algebraic formulations of other knotted surfaces in Yoshikawa's table are summarized in Table\;\ref{tab01}.


\renewcommand{\arraystretch}{1.74}
\begin{center}
\small
$\begin{array}{|l|l|}
  \hline 
	\text{{\bf Name(s) of knotted surface}}&\text{{\bf Surface singular braid form}}\\
	\hline
	0_1, \text{ unknotted }\mathbb{S}^2 & [1]_1\\
	\hline
	2_1^{1}, \text{ unknotted } \mathbb{T}^2& [ab]_2\\
	\hline
	2_1^{-1}, \text{ unknotted } \mathbb{P}^2_+ & [ac_1]_2\\
  \hline
	\text{unknotted } \mathbb{P}^2_- & [bc_1]_2\\
  \hline
	\text{unknotted } \mathbb{KB}^2 & [abc_1]_2\\
	\hline
  7_1^{0,-2} & [abc_2^{-1}c_1^{-2}c_2^{-1}c_1^{-1}]_3\\
  \hline
	10_2^{0,1}& [ab(c_2c_1^2c_2)^2]_3\\
	\hline
	\tau^n(\text{rational link }C[k_1,k_2,\ldots,k_{2m+1}])& [ac_2^{k_{2m+1}}c_1^{-k_{2m}}\cdots c_2^{k_1}bc_2^{-k_1}c_1^{k_2}\cdots c_2^{-k_{2m+1}}(c_1c_2c_1)^{2n}]_3\\
  \hline
	8_1^{1,1}, \text{ spun surface of Hopf link}& [(abc_2^{-1}c_3^{-1}c_1c_2)^2]_4\\
  \hline
	8_1^{-1,-1} & [bc_2^{-1}c_1^{-1}c_2c_1^2c_3^{-1}c_2bc_2^{-1}c_1^{-1}c_2^{-1}c_1^2c_3^{-1}c_2]_4\\
	\hline
	9_1^{0,1} & [abc_2^{-1}c_3^{-1}c_2^2c_3^{-1}c_2c_1^2c_2]_4\\
  \hline
	9_1^{1,-2} & [(abc_2^{-1}c_3^{-1}c_1c_2)^2c_1^{-1}]_4\\
  \hline
	10_1^{1}, \text{ spun torus of the trefoil}& [ac_2c_3c_1c_2bc_2^{-1}c_1^{-1}c_3^{-1}c_2^3c_3c_1c_2ac_2^{-1}c_1^{-1}c_3^{-1}c_2^{-1}bc_2^{-3}]_4\\
  \hline
	10_1^{1,1} & [ac_2^{-1}c_1^{-1}c_3c_2^{-1}bc_2^{-1}c_1^2c_2ac_2^{-1}c_1^{-1}c_3^{-1}c_2bc_2^2]_4\\
  \hline
	10_1^{0,0,1} & [abc_3c_2^{-1}c_1^{-2}c_2^{-1}c_3^{-1}c_2c_1^2c_2]_4\\
  \hline
	10_1^{0,-2} & [abc_2^{-1}c_1^{-1}c_2^{-1}c_1^{-1}c_3c_2^{-2}c_3]_4\\
  \hline
	10_2^{0,-2} & [ac_2^{-1}c_3c_2c_1^{-1}c_2^2bc_2^{-1}c_1^{-1}c_3^{-1}c_2^{-1}c_1c_2^{-1}]_4\\
  \hline
	10_1^{-1,-1}& [ac_2^{-1}c_1^{-1}c_3^{-1}c_2bc_2c_1^{-1}c_2c_3c_2^{-1}c_1c_2^{-1}c_1^2c_2]_4\\
	\hline
	10_1^{-2,-2}& [(abc_2^{-1}c_3^{-1}c_1c_2c_1^{-1})^2]_4\\
	\hline
	9_1 & [ac_2^{-1}c_1^{-1}c_3^{-1}c_2^{-1}c_4^{-1}c_3^{-2}c_2c_1^{-1}c_3c_2c_3c_4c_2c_3c_1c_2bc_2^{-1}\cdot\\
	    & \cdot c_1^{-1}c_3^{-1}c_2^{-1}c_4^{-1}c_5^{-1}c_6^{-1}c_4^{-1}c_5c_7^{-1}c_6^{-1}c_4^{-1}c_3^{-1}c_4c_2^{-1}c_3^{-1}\cdot\\
			& \cdot c_1c_2^{-1}c_4c_5c_4c_3c_4c_5^{-1}c_6c_5^{-1}c_4^{-1}c_7c_6c_5c_3c_4c_2c_3c_1c_2]_8\\
\hline
	\end{array}$
	\captionof{table}{Surface singular braid formulations of knotted surfaces.\label{tab01}}
\end{center}
\renewcommand{\arraystretch}{1}

\begin{prop}\label{prop5}
In the monoid $SSB_n$ for $n\in \mathbb{Z}$ and $n>1$ the following relations hold.
\begin{align}
abc_1&\not=ab, \tag{e1}\\
ac_1^2&\not=a, \tag{e2}\\
bc_1^2&\not=b, \tag{e3}\\
ac_2&\not=c_2a& \text{\;\; for } n>2, \tag{e4}\\
bc_2&\not=c_2b& \text{\;\; for } n>2, \tag{e5}\\
c_1c_2&\not=c_2c_1& \text{\;\; for } n>2, \tag{e6}\\
(c_1c_2c_1)^2&\not=1& \text{\;\; for } n>2. \tag{e7}
\end{align}
\end{prop}

\begin{proof}
For elements of the monoid $SSB_n$ it follows that (see Thm.\;\ref{thm4} and Prop.\;\ref{prop1}):
$$[abc_1]_n=\mathbb{KB}^2\sqcup\underbrace{\mathbb{S}^2\sqcup\ldots\sqcup\mathbb{S}^2}_{n-2}\not=\mathbb{T}^2\sqcup\underbrace{\mathbb{S}^2\sqcup\ldots\sqcup\mathbb{S}^2}_{n-2}=[ab]_n,$$
$$[ac_1]_n=[bc_1^{-1}]_n=\mathbb{P}^2_+\sqcup\underbrace{\mathbb{S}^2\sqcup\ldots\sqcup\mathbb{S}^2}_{n-2}\not=\mathbb{P}^2_-\sqcup\underbrace{\mathbb{S}^2\sqcup\ldots\sqcup\mathbb{S}^2}_{n-2}=[bc_1]_n=[ac_1^{-1}]_n.$$
This implies the relations (e1)-(e3). Consider now the spun $2$-knot of the trefoil, it is the well known nontrivial $2$-knot, as its group is isomorphic to the group of the classical trefoil. It follows from Prop.\;\ref{prop1} that this knotted $2$-sphere can be presented as $\tau^0(T(2,3))=[ac_2^{-3}bc_2^{3}]_3$, we also have trivial $2$-sphere $\tau^1(T(2,3))=[ac_2^{-3}bc_2^{3}(c_1c_2c_1)^{2}]_3$ (see \cite{Zee65} for the proof), therefore we have

$$\mathbb{T}^2\sqcup\bigsqcup_{n-2}\mathbb{S}^2=[ab]_n\not=[ac_2^{-3}bc_2^{3}]_n\not=[ac_2^{-3}bc_2^{3}(c_1c_2c_1)^{2}]_n\not=[abc_1c_2]_n=\mathbb{KB}^2\sqcup\bigsqcup_{n-3}\mathbb{S}^2.$$
This, together with the relations (m1), (m3)-(m7), (m14), (m15) implies the relations (e4)-(e7).
\end{proof}

\section{Representations}

Let $K$ throughout this paper denote a field. By a \emph{representation} of a monoid $D$ of dimension $n$ over $K$ we mean a homomorphism $\rho$ of $D$ into the multiplicative monoid of $M_n(K)$ of all $n\times n$ matrices with entries in $K$. If $\rho$ is injective then the representation is said to be \emph{faithful}. Denote $I_t$ and $0_t$ the identity matrix and the zero matrix of size $t\times t$ respectively.
\begin{prop}\label{prop6}
For $n,m\geqslant 2$ and any faithful representation $\phi: SSB_n\to M_m(K)$, there is a faithful representation $\rho: SSB_n\to M_m(K)$ such that:

\begin{align}
\rho(a)&=I_s\oplus 0_{m-s}& \text{\;\; where } s\in\{1,\ldots, m-1\}, \tag{p1}\\
\rho(b)&\not\in\{0_m,I_m\}, \tag{p2}\\
\rho(a)&\not=\rho(b), \tag{p3}\\
\rho(a)\rho(b)&\not=\rho(a), \tag{p4}\\
\rho(a)\rho(b)&\not=\rho(b), \tag{p5}\\
\rho(a)\rho(b)\rho(c_1)&\not=\rho(a)\rho(b), \tag{p6}\\
\rho(a)\rho(c_2)&\not=\rho(c_2)\rho(a)& \text{\;\; for } n>2, \tag{p7}\\
\rho(b)\rho(c_2)&\not=\rho(c_2)\rho(b)& \text{\;\; for } n>2, \tag{p8}\\
\rho(c_1)\rho(c_2)&\not=\rho(c_2)\rho(c_1)& \text{\;\; for } n>2, \tag{p9}\\
(\rho(c_1)\rho(c_2)\rho(c_1))^2&\not=I_m& \text{\;\; for } n>2, \tag{p10}\\
\rho(a)\rho(c_1)^2&\not=\rho(a), \tag{p11}\\
\rho(b)\rho(c_1)^2&\not=\rho(b). \tag{p12}
\end{align}
\end{prop}
Let us recall the following property of idempotent matrix.
\begin{lem}\label{lem1}
If a matrix $X$ with entries in a field $K$ satisfies $X^2=X$ then it is diagonalizable and all its eigenvalues are either $0$ or $1$.
\end{lem}
\begin{proof}
Consider $X$ as an endomorphism operator on a vector space $V$. Take any nonzero vector $u\in \text{im}X$, then there exists $v\in V$ such that $Xv=u$, from the idempotency relation $X^2=X$ we have $u=Xv=XXv=Xu$ which yields $u\not\in \text{ker}X$, so we have $V=\text{im}X\oplus\text{ker}X$, therefore $X$ is diagonalizable. If $\lambda$ is its eigenvalue then there exists nonzero vector $v\in V$ such that $\lambda v=Xv=X^2v=X\lambda v=\lambda^2 v$. We must have then that $\lambda(\lambda-1)=0$, and because $K$ is a field, this implies $\lambda\in\{0,1\}$.
\end{proof}
\begin{proof}[Proof of Prop.\;\ref{prop6}]
The monoid $M_m(K)$ of $m\times m$ matrices over $K$ can be identified with $End_K(V)$, the monoid of endomorphisms of a vector space $V$ over $K$ of finite dimension $m$. Applying Lem.\;\ref{lem1} for $X=\phi(a)$, we can conclude that there exists a matrix $P\in GL_m(K)$ such that $P^{-1}\phi(a)P=I_s\oplus 0_{m-s}$, where $s\in\{0,\ldots, m\}$. We define a new representation by setting $\rho(x)=P^{-1}\phi(x)P$ for any $x\in SSB_n$, and now its injectivity follows immediately from injectivity of $\phi$. It proves (p1) beside the cases $s=0$, $s=m$ which will be excluded later.

From Prop.\;\ref{prop5} we have $abc_1\not=ab$, $ac_1^2\not=a$ and $bc_1^2\not=b$, and together with the relations (m1) and (m12)-(m14) we moreover have $a\not=b$, $b\not=1$, $a\not=1$, hence from injectivity of $\rho$ we have the cases (p2), (p3), (p6), (p11), (p12) and remaining cases $s=0, s=m$ from (p1). The relations (p4) and (p5) follow from (m14) together with (p11) and (p12) respectively.
The remaining relations (p7)-(p10) follow directly from (e4)-(e7).
\end{proof}
\begin{prop}\label{prop4}
If a representation $\rho: SSB_n\to M_m(K)$ for $n,m\geqslant 2$ satisfies $\text{rank}(\rho(a))=1$ or $\text{rank}(\rho(b))=1$, then $\rho$ is not faithful.
\end{prop}
\begin{proof}
From the symmetric role of $a$ and $b$ in $SSB_n$, we can assume that $\text{rank}(\rho(a))=1$. Denote $A=\rho(a), B=\rho(b)$ and $B=(b_{i,j})_{i,j\in\{1,\ldots,m\}}$. From the relation (p1) of Prop.\;\ref{prop6} we can assume that $A=I_1\oplus 0_{m-1}$, then from $AB=BA$ it follows that $b_{1,2}=\cdots=b_{1,m}=0$ and $b_{2,1}=\cdots=b_{m,1}=0$. This implies that $AB=(b_{1,1})\oplus 0_{m-1}$, and combining it with $B^2=B$ gives us the relation $AB=A$ that contradicts the relation (p4) of Prop.\;\ref{prop6}.
\end{proof}
From Prop.\;\ref{prop6} and Prop.\;\ref{prop4} we immediately have the following.
\begin{cor}
No representation $\rho: SSB_n\to M_2(K)$ for $n\geqslant 2$ is faithful.
\end{cor}

\begin{exmp}
A faithful representation $\rho$ of the monoid $SSB_2$ can be defined (in a field of characteristic zero) as follows:
$$\rho(a)=\left(\begin{array}{ccc}1 & 0 & 0\\ 0 & 1 & 0 \\ 0 & 0 & 0\end{array}\right),\;\; \rho(b)=\left(\begin{array}{ccc}0 & 0& 0\\ 0 & 1 & 0\\ 0 & 0 & 1\end{array}\right),\;\; \rho(c_1)=\left(\begin{array}{ccc}2 & 0 & 0 \\0 & -1 & 0 \\0 & 0 & 2 \end{array}\right).$$
\end{exmp}

\begin{thm}\label{thm1}
No representation $\rho: SSB_n\to M_3(K)$ for $n\geqslant 3$ is faithful.
\end{thm}
\begin{proof}
Assume the contrary, that $\rho$ is faithful and denote $X:=\rho(x)$ for $x\in\{a,b,c_1\}$. From the relations (m11) and (m13) of Prop.\;\ref{prop3} we have $B^2=B$ and $AB=BA$. From Prop.\;\ref{prop6} we can assume that $B\not\in\{0_3,I_3\}$, $AB\not=A$, $AB\not=B$, $ABC_1\not=AB$ and that $A=I_2\oplus 0_1$. Let $B=(b_{i,j})_{i,j\in\{1,2,3\}}$. Then from the relation $AB=BA$ it follows that $B=G\oplus(b_{3,3})$ for some matrix $G\in M_2(K)$. From the relation $B^2=B$ it follows that $G^2=G$ and $b_{3,3}\in\{0,1\}$. If $b_{3,3}=0$ then $AB=G\oplus0_1=B$, a contradiction, so it follows that $b_{3,3}=1$, and combining it with $B\not=I_3$ gives moreover $\text{det}{G}=0$. 

Consider now $C_1=(c_{i,j})_{i,j\in\{1,2,3\}}$, from the the relation (m4) it follows that $C_1=F\oplus(c_{3,3})$ for some matrix $F\in M_2(K)$. Non-invertability of matrix $G$ together with the relation $ABC_1\not=AB$ implies that $G\not\in\{0_2,I_2\}$ and $GF\not=G$, additionally from the relations (m1), (m5) and (m14) we have $\text{det}{F}\not=0$, $GF=FG$ and $GF^2=G$. Consider the following two main cases.

{\bf Case (a)}. Assume $b_{1,1}b_{2,2}b_{1,2}b_{2,1}=0$. Then by $\text{det}{G}=0$ we must have that $b_{1,2}b_{2,1}=0$. From the symmetric role of $b_{1,2}$ and $b_{2,1}$, without loss of generality, assume $b_{1,2}=0$. Then from $G^2=G$ and $G\not\in\{0_2,I_2\}$ we obtain that $G$ is one of the two possible forms: $\left(\begin{array}{cc}0 & 0\\ b_{2,1} & 1\end{array}\right)$ or $\left(\begin{array}{cc}1 & 0\\ b_{2,1} & 0\end{array}\right)$. From $GF=FG$ it follows that $c_{1,2}=0$, and the relation $GF^2=G$ together with $GF\not=G$ yields that $F$ is one of the two possible forms $\left(\begin{array}{cc}c_{1,1} & 0\\-b_{2,1}(1+c_{1,1}) & -1\end{array}\right)$ or $\left(\begin{array}{cc}-1 & 0\\-b_{2,1}(1+c_{2,2}) & c_{2,2}\end{array}\right)$ one for each mentioned type of $G$ respectively. Additionally from (m1) and (p11) we have that $c_{1,1}\not=0$, $c_{1,1}^2\not=1$ in the first and $c_{2,2}\not=0$, $c_{2,2}^2\not=1$ in the second case respectively. From (m1) and (p12) we moreover have $c_{3,3}\not=0$ and $c_{3,3}^2\not=1$ in both cases.

{\bf Case (b)}. Assume $b_{1,1}b_{2,2}b_{1,2}b_{2,1}\not=0$. Then from the relations (m1), (m4), (m5), (m11), (m13), (m14), (p2), (p4)-(p6), (p11) and (p12) it follows that $B,C_1$ are in the form: 
$$B=\left(\begin{array}{cc}b_{1,1}&b_{1,2}\\\frac{(1-b_{1,1})b_{1,1}}{b_{1,2}}&1-b_{1,1}\end{array}\right)\oplus I_1,$$$$ C_1=\left(\begin{array}{cc}\frac{(b_{1,1}-1)c_{1,2}-b_{1,2}}{b_{1,2}}&c_{1,2}\\\frac{(1-b_{1,1})b_{1,1}c_{1,2}}{b_{1,2}^2}&-\frac{b_{1,1}c_{1,2}}{b_{1,2}}-1\end{array}\right)\oplus(c_{3,3}),$$
for $b_{1,1}b_{1,2}c_{1,2}c_{3,3}\not=0$, $c_{1,2}\not=-b_{1,2}$, $c_{1,2}\not=-2b_{1,2}$, $b_{1,1}\not=1$ and $c_{3,3}^2\not=1$.

Introducing now a matrix $\rho(c_2)$ and making simple, but tedious computations (summarized in the following Appendix) we can show that in both of the above cases, the relations (m1), (m3), (m6), (m7), (m10), (m15), (p7), (p8) and (p10) form a self-contradictory set.

\end{proof}

\begin{remark}
For $n\in \mathbb{Z}$ and $n>1$ we define monoid $\overline{SSB_n}$ by the following presentation:
$$\overline{SSB_n}=\left<a,b \text{ and } c_i,c_i^{-1} \text{ for } i=1,\ldots,n-1 \;|\; (m1)-(m16), R\right>,$$
where $R$ is any finite sequence of relations, on generators of $\overline{SSB_n}$, that follow from the topological Yoshikawa moves.
\begin{itemize}
	\item[(i)] In $\overline{SSB_n}$ relations \emph{(e1)-(e7)} must hold.
  \item[(ii)] If a representation $\rho: \overline{SSB_n}\to M_m(K)$ for $n,m\geqslant 2$ satisfies $\text{rank}(\rho(a))=1$ or $\text{rank}(\rho(b))=1$ then $\rho$ is not faithful.
	\item[(iii)] No representation $\rho: \overline{SSB_n}\to M_2(K)$ for $n\geqslant 2$ is faithful.
	\item[(iv)] No representation $\rho: \overline{SSB_n}\to M_3(K)$ for $n\geqslant 3$ is faithful.
\end{itemize}
\end{remark}

\newpage

\section*{Appendix}
\noindent
Let $\rho(c_2)=(d_{i,j})_{i,j\in\{1,2,3\}}$, we consider further subcases.
\\
{\bf Case (a1)}. Assume $b_{1,1}=0$, $b_{2,2}=1$, $c_{2,2}=-1$ and $c_{2,1}=-b_{2,1}(1+c_{1,1})$. Consider further two cases.
\\
{\bf Case (a1a)}. Assume $d_{3,3}=0$. Consider further two cases.
\\
{\bf Case (a1a1)}. Assume $d_{3,2}=0$. From (m1) and (m3) we have $d_{3,1}\not=0$, $d_{1,1}=c_{3,3}$, $d_{1,2}=0$ and $d_{1,3}=0$, a contradiction with (m1).
\\
{\bf Case (a1a2)}. Assume $d_{3,2}\not=0$. From (m1) and (m10) we have $d_{3,1}=b_{2,1} d_{3,2}$, now (m3) with (m6) contradict (m1).
\\
{\bf Case (a1b)}. Assume $d_{3,3}\not=0$. Consider further two cases.
\\
{\bf Case (a1b1)}. Assume $d_{3,2}=0$. From (m1), (m3) and (m10) we have $d_{3,1}=0$, now (m1) with (m3) yield $d_{3,3}=c_{3,3}$. Consider further two cases.
\\
{\bf Case (a1b1a)}. Assume $d_{1,2}=0$. From (m10) we have $d_{1,3}=0$, from (m3) it follows that $d_{1,1}=c_{1,1}$ and $d_{2,2}=-1$. Now (m6) with (p7) contradict (p10).
\\
{\bf Case (a1b1b)}. Assume $d_{1,2}\not=0$. From (m15) we have $d_{2,3}=\frac{1}{d_{1,2}}(d_{1,3}d_{2,2}-c_{3,3}^3d_{1,3})$ and $d_{2,1}=\frac{1}{d_{1,2}}(\frac{1}{c_{1,1}^2}+d_{1,1}d_{2,2})$. From (p7) we have $d_{1,3}\not=0$, now (m3) with (m6) contradict (m15).
\\
{\bf Case (a1b2)}. Assume $d_{3,2}\not=0$. From (m6) we have $d_{2,1}=b_{2,1}c_{1,1}^2d_{1,1}-b_{2,1}d_{1,1}-\frac{c_{1,1}^2d_{1,1}d_{3,1}}{d_{3,2}}-\frac{c_{3,3}^2 d_{3,1}d_{3,3}}{d_{3,2}}$ and $d_{2,2}=b_{2,1}c_{1,1}^2d_{1,2}-b_{2,1}d_{1,2}-\frac{c_{1,1}^2d_{1,2}d_{3,1}}{d_{3,2}}-c_{3,3}^2d_{3,3}$. From (m3) we have $d_{2,3}=\frac{1}{d_{3,2}}(-b_{2,1}c_{1,1}d_{1,3}d_{3,2}-b_{2,1}d_{1,3}d_{3,2}-c_{3,3}^2d_{3,3}+c_{3,3} d_{3,3}^2+c_{1,1} d_{1,3} d_{3,1})$. Consider further two cases.
\\
{\bf Case (a1b2a)}. Assume $d_{3,1}=0$. Form (m1) and (m7) we have $d_{1,3}=\frac{d_{1,2}}{d_{3,2}}(d_{3,3}-\frac{c_{1,1}^2d_{1,1}}{c_{3,3}^2})$. From (m3) and (m10) we obtain $d_{3,3}=\frac{-1}{c_{3,3}+1}$ and $d_{1,1}=c_{1,1}$. From (m3), (m6), (m10) and (p8) we have $b_{2,1}=0$, $c_{1,1}^3=1$ and $c_{3,3}^3=1$, a contradiction with (p10).
\\
{\bf Case (a1b2b)}. Assume $d_{3,1}\not=0$. Consider further two cases.
\\
{\bf Case (a1b2b1)}. Assume $d_{1,2}=0$. From (m3) we have $d_{1,3}=0$, $d_{1,1}=c_{1,1}$ and $d_{3,3}=\frac{-1}{c_{3,3}+1}$. Now (m15) contradicts (p10).
\\
{\bf Case (a1b2b2)}. Assume $d_{1,2}\not=0$. Form (m1), (m7) and (m10) we have $b_{2,1}=\frac{d_{3,1}}{d_{3,2}}$ and $d_{1,1}=-\frac{c_{3,3}^2d_{1,3}d_{3,2}}{c_{1,1}^2d_{1,2}}+\frac{c_{3,3}^2d_{3,3}}{c_{1,1}^2}+\frac{d_{1,2} d_{3,1}}{d_{3,2}}$. From (m3) we have $d_{3,3}=\frac{-1}{c_{3,3}+1}$, and (m15) yields $c_{3,3}^2+c_{3,3}=-1$. This together with (m10) implies $d_{1,3}=0$, a contradiction with (p10).
\\
{\bf Case (a2)}. Assume $b_{1,1}=1$, $b_{2,2}=0$, $c_{1,1}=-1$ and $c_{2,1}=-b_{2,1}(1+c_{2,2})$. Consider further two cases.
\\
{\bf Case (a2a)}. Assume $d_{3,2}=0$. Consider further two cases. 
\\
{\bf Case (a2a1)}. Assume $d_{1,2}=0$. From (m1) and (m3) we have $d_{2,2}=c_{2,2}$. Consider further two cases. 
\\
{\bf Case (a2a1a)}. Assume $d_{1,3}=0$. From (m1) and (m3) we have $d_{1,1}=-1$, $d_{3,3}=c_{3,3}$, from (m10) we obtain $d_{2,3}=0$. From (p7) and (m6) we obtain $c_{3,3}^3=1$, from (p8) and (m7)  we obtain $c_{2,2}^3=1$, a contradiction with (p10). 
\\
{\bf Case (a2a1b)}. Assume $d_{1,3}\not=0$. From (m3) we have $d_{1,1}=c_{3,3}(d_{3,3}+1)$ and $d_{3,1}=-\frac{c_{3,3}d_{3,3}(c_{3,3}-d_{3,3})}{d_{1,3}}$, from (m6) we have $d_{3,3}=\frac{-1}{c_{3,3}+1}$ and $d_{1,1}=\frac{c_{3,3}^2}{c_{3,3}+1}$. From (m15) we have $c_{3,3}^3=1$, then from (m6) and (m10) we have $d_{2,1}=b_{2,1}(-c_{2,2}+c_{3,3}+\frac{1}{c_{3,3}+1}-1)$ and $d_{2,3}=b_{2,1} d_{1,3}$, a contradiction with (p8).
\\
{\bf Case (a2a2)}. Assume $d_{1,2}\not=0$. From (m6) we have $d_{3,1}=0$, and (m7) yields $d_{1,1}=b_{2,1} c_{2,2}^2 d_{1,2}-b_{2,1} d_{1,2}-c_{2,2}^2 d_{2,2}$. From (m1) and (m3) it follows that $d_{3,3}=c_{3,3}$, from (m6) we have $d_{2,3}=\frac{c_{2,2}^2d_{1,3}d_{2,2}-c_{3,3}^3d_{1,3}}{c_{2,2}^2d_{1,2}}.$ From (p7) we have $d_{1,3}\not=0$, from (m6) it follows that $d_{2,1}=\frac{b_{2,1} c_{2,2}^4d_{1,2}d_{2,2}-b_{2,1}c_{2,2}^2d_{1,2}d_{2,2}-c_{2,2}^4d_{2,2}^2+c_{3,3}^6}{c_{2,2}^2 d_{1,2}}.$ From (m15) it follows that $c_{2,2}^3=1$, a contradiction with (p10).
\\
{\bf Case (a2b)}. Assume $d_{3,2}\not=0$. Then from (m3) we have the relations\\
$d_{2,1}=\frac{d_{1,1} \left(b_{2,1} \left(c_{2,2}+1\right) d_{3,2}+d_{3,1}\right)-c_{3,3} \left(b_{2,1} \left(c_{2,2}+1\right) d_{3,2}+d_{3,1} \left(d_{3,3}+1\right)\right)}{c_{2,2} d_{3,2}}$, \\
$d_{2,2}=\frac{d_{1,2} \left(b_{2,1} \left(c_{2,2}+1\right) d_{3,2}+d_{3,1}\right)+c_{3,3} d_{3,2} \left(c_{2,2}-d_{3,3}\right)}{c_{2,2} d_{3,2}}$ and \\
$d_{2,3}=\frac{d_{1,3} \left(b_{2,1} \left(c_{2,2}+1\right) d_{3,2}+d_{3,1}\right)+c_{3,3} d_{3,3} \left(c_{3,3}-d_{3,3}\right)}{c_{2,2} d_{3,2}}.$ Consider further two cases. 
\\
{\bf Case (a2b1)}. Assume $d_{3,1}=0$. Consider further two cases. 
\\
{\bf Case (a2b1a)}. Assume $b_{1,2}=0$. Then (m10) implies $c_{2,2}=0$, a contradiction with (m1). 
\\
{\bf Case (a2b1b)}. Assume $b_{1,2}\not=0$.  Then (m10) implies $d_{1,3}=\frac{b_{2,1} d_{1,2} d_{3,3}+d_{1,1} d_{3,3}}{b_{2,1} d_{3,2}}$ again from (m10) we obtain $d_{3,2}=0$, a contradiction. 
\\
{\bf Case (a2b2)}. Assume $d_{3,1}\not=0$. Consider further two cases. 
\\
{\bf Case (a2b2a)}. Assume $c_{2,2}=c_{3,3}$. Then (m1) and (m15) yields $d_{1,1}=\frac{-b_{2,1} c_{2,2}^2 d_{1,2}-b_{2,1}c_{2,2}d_{1,2}+1}{c_{2,2}(c_{2,2}+1)}$, later from (m10) we obtain\\ $b_{2,1}=\frac{c_{2,2}^2 \left(-d_{1,3}\right) d_{3,1}-c_{2,2} d_{1,3} d_{3,1}+d_{3,3}}{\left(c_{2,2}^2+c_{2,2}\right) d_{1,3} d_{3,2}}$. Again from (m10) we obtain $d_{3,3}=0$, a contradiction with (m1). 
\\
{\bf Case (a2b2b)}. Assume $c_{2,2}\not=c_{3,3}$. Then from (m6) we have \\
$d_{3,3}=\frac{\left(c_{2,2}+1\right) d_{1,2} \left(b_{2,1} d_{3,2}+d_{3,1}\right)+c_{3,3} c_{2,2}^2 d_{3,2}}{\left(c_{2,2}-c_{3,3}\right) c_{3,3} d_{3,2}},$ from (m3) we can obtain\\
$d_{1,3}=\frac{d_{1,2}(d_{1,2}(b_{2,1}(c_{2,2}+1)d_{3,2}+(c_{3,3}+1)d_{3,1})+d_{3,2}(c_{2,2}(c_{3,3}^2+c_{3,3}+d_{1,1})-c_{3,3} d_{1,1}-c_{2,2}^2))}{(c_{2,2}-c_{3,3})c_{3,3}d_{3,2}^2}.$\\ Then from (m3) we have $d_{3,1}=-b_{2,1} \left(c_{2,2}+1\right) d_{3,2}.$ Consider further two cases. 
\\
{\bf Case (a2b2b1)}. Assume $b_{1,2}=0$. Then (m10) implies $c_{2,2}=0$, a contradiction with (m1). 
\\
{\bf Case (a2b2b2)}. Assume $b_{1,2}\not=0$. Then (m10) implies $b_{2,1}=\frac{c_{3,3}^2}{\left(c_{2,2}+1\right) d_{1,2}}$, again from (m10) we have $c_{3,3}=-1$, a contradiction.
\\
{\bf Case (b1)}. Assume $d_{3,3}=0$. Consider further two cases.
\\
{\bf Case (b1a)}. Assume $d_{3,2}=0$. From (m1) we have $d_{3,1}\not=0$, from (m10) we have $d_{2,3}=0$ and $d_{1,3}=0$, a contradiction with (m1).
\\
{\bf Case (b1b)}. Assume $d_{3,2}\not=0$. From (m1) and (m10) we have $d_{3,1}=\frac{\left(b_{1,1}-1\right)d_{3,2}}{b_{1,2}}$ or $d_{3,1}=\frac{b_{1,1}d_{3,2}}{b_{1,2}}$, in both of those cases (m3) contradicts (m1).
\\
{\bf Case (b2)}. Assume $d_{3,3}\not=0$. From (m10) we have $d_{2,1}=\frac{\left(b_{1,1}-1\right)d_{2,2}}{b_{1,2}}$, $d_{3,2}=\frac{b_{1,2}d_{3,1}}{b_{1,1}}$ and $d_{2,3}=0$. From (m1) we have $d_{2,2}\not=0$ then (m6) yields $d_{1,3}=0$, the relation (p7) yields $d_{3,1}\not=0$, and (m3) yields $d_{3,3}=c_{3,3}$. Now (m3) with (m6) contradict (p10).

\end{document}